\documentclass[11pt,a4]{article}
\usepackage{latexsym}
\usepackage{amsmath}
\usepackage{amssymb}
\usepackage{amsthm}
\usepackage{amscd}

\newtheorem{theorem}{Theorem}[section]
\newtheorem{lemma}[theorem]{Lemma}
\newtheorem{corollary}[theorem]{Corollary}
\newtheorem{proposition}[theorem]{Proposition}

\theoremstyle{definition}

\newtheorem{remark}[theorem]{Remark}

\theoremstyle{remark}

\newcommand{\N}{{\mathbb N}}
\newcommand{\R}{{\mathbb R}}

\newcommand{\Z}{{\mathbb Z}}

\begin{document}

\title{A note on the boundedness\\
of discrete commutators on Morrey spaces and their preduals
}

\author{Yoshihiro Sawano}

\maketitle

\begin{abstract}
Dyadic fractional integral operators
are shown to  be bounded
on Morrey spaces and their preduals.
It seems that the proof of the boundedness
by means of dyadic fractional integral operators
is effective particularly on the preduals.
In the present paper
the commutators are proved to be bounded as well.
\end{abstract}

\section{Introduction}

In the present paper,
we consider the dyadic analysis of Morrey spaces and their preduals.
The Haar wavelet, which plays a central role in this field, 
is given as follows:
First, we write
\begin{equation}
h^{\varepsilon_i}(t)
:=\chi_{[0,1)}(2t)
+(-1)^{\varepsilon_i}\chi_{[1,2)}(2t) \quad (t \in {\mathbb R})
\end{equation}
for $\varepsilon_i \in {\mathbb Z}/2{\mathbb Z}$.
Given $\varepsilon \in E:=({\mathbb Z}/2{\mathbb Z})^n 
\setminus \{(0,0,\ldots,0)\}$,
we define
\begin{equation}
h^\varepsilon
:=h^{\varepsilon_1}\otimes 
h^{\varepsilon_2}\otimes \ldots \otimes h^{\varepsilon_n},
\mbox{ that is, }
h^\varepsilon(x_1,x_2,\ldots,x_n)
=\prod_{i=1}^n
h^{\varepsilon_i}(x_i).
\end{equation}
By ${\mathcal D}$ we mean the set of all dyadic cubes.
If we write
$\displaystyle Q_{j m}:=
\prod_{\nu=1}^n \left[\frac{m_\nu}{2^j},\frac{m_{\nu}+1}{2^j}\right)$
for $j \in {\mathbb Z}$ and $m \in {\mathbb Z}^n$,
then we have 
${\mathcal D}=\{Q_{jm}\,:\,j \in {\mathbb Z}, \, m \in {\mathbb Z}^n\}$.
The set ${\mathcal D}_j$ is the subset of ${\mathcal D}$
made up of the cubes of volume $2^{-jn}$:
${\mathcal D}_j=\{Q_{jm}\,:\, m \in {\mathbb Z}^n\}$.
Given a dyadic cube 
$Q=Q_{jm} \quad (j \in {\mathbb Z}, \, m \in {\mathbb Z}^n)$,
we define the corresponding Haar function by
\begin{equation}
h^\varepsilon_Q(x):=2^{jn/2}h^{\varepsilon}(2^j x-m).
\end{equation}
The idea of discretizing $I_\alpha$ dates back to Lacey's 2007 paper
\cite{Lacey-2005-Hokkaido}.

Now we will describe Morrey spaces,
the function spaces considered in the present paper.
Let $1 \le q \le p<\infty$.
Then let us define the Morrey norm $\|f\|_{{\mathcal M}^p_q}$ by
\begin{equation}\label{eq:Morrey}
\|f\|_{{\mathcal M}^p_q}
:=
\sup_{Q \in {\mathcal D}}|Q|^{\frac1p-\frac1q}
\left(\int_Q|f(y)|^q\,dy\right)^{\frac1q},
\end{equation}
where $f \in L^{q,\rm loc}$.
We will also use the dyadic BMO space.
Given a cube $Q \in {\mathcal D}$ and $f \in L^{1,{\rm loc}}$,
we can write $\displaystyle m_Q(f):=\frac{1}{|Q|}\int_Q f(x)\,dx$.
The dyadic sharp maximal operator here is defined by
\begin{equation}
M^{\sharp,{\rm dyadic}}f(x)
:=
\sup_{x \in {\mathcal D}}
m_Q(|f-m_Q(f)|).
\end{equation}
A function $a \in L^{1,{\rm loc}}$ 
is said to belong to the dyadic BMO,
which we will write as ${\rm BMO}_{{\rm dyadic}}$,
if $M^{\sharp,{\rm dyadic}}a \in L^\infty$.
We define the dyadic BMO norm
by $\|a\|_{{\rm BMO}_{{\rm dyadic}}}:=\|M^{\sharp,{\rm dyadic}}a\|_\infty$.

The present paper,
based upon Theorem \ref{thm1},
considers the boundedness of commutators.
Throughout the paper, for $A,B>0$,
we write $A \lesssim B$ to indicate
that there exists a constant $c>1$
such that $A \le c\,B$
and that this constant depends only on $p,q,s,t,\alpha$
which will appear in each theorem.
We also use $A \gtrsim B$ to denote $B \lesssim A$
and $A \sim B$ to denote the two-sided inequality
$A \lesssim B \lesssim A$.
\begin{theorem}
\label{thm1}
Let $1<q \le p<\infty$.
\begin{enumerate}
\item[$(i)$]
Let $f \in {\mathcal M}^p_q$.
Then we have equivalence
\begin{equation}
\label{eq:thm1-1}
\|f\|_{{\mathcal M}^p_q}
\sim
\sum_{\varepsilon \in E}
\left\|\left(\sum_{j=-\infty}^\infty\left|\sum_{Q \in {\mathcal D_j}}
\langle f,h_Q^\varepsilon \rangle h_Q^{\varepsilon}\right|^2
\right)^{\frac12}\right\|_{{\mathcal M}^p_q}.
\end{equation}
\item[$(ii)$]
If a locally integrable function $f$
satisfies
\begin{equation}
\label{eq:thm1-2}
\sum_{\varepsilon \in E}
\left\|\left(\sum_{j=-\infty}^\infty\left|\sum_{Q \in {\mathcal D_j}}
\langle f,h_Q^\varepsilon \rangle h_Q^{\varepsilon}\right|^2
\right)^{\frac12}\right\|_{{\mathcal M}^p_q}<\infty,
\end{equation}
then the limit
\begin{equation}
\label{eq:thm1-3}
g:=\lim_{M \to \infty}
\sum_{\varepsilon \in E}
\sum_{j=-M}^M
\sum_{Q \in {\mathcal D}_j}
\langle f,h^{\varepsilon}_Q \rangle h^{\varepsilon}_Q
\end{equation}
exists in the topology of $L^{q,{\rm loc}}$
and defines an ${\mathcal M}^p_q$-function.
Furthermore,
\[
\|g\|_{{\mathcal M}^p_q}
\sim
\sum_{\varepsilon \in E}
\left\|\left(\sum_{j=-\infty}^\infty\left|\sum_{Q \in {\mathcal D_j}}
\langle f,h_Q^\varepsilon \rangle h_Q^{\varepsilon}\right|^2
\right)^{\frac12}\right\|_{{\mathcal M}^p_q}.
\]
\end{enumerate}
\end{theorem}

The following paraproduct plays an important role
in the proof of the boundedness of commutators.
The next result follows.
\begin{theorem}
\label{thm2}
Let $a \in {\rm BMO}_{{\rm dyadic}}$ and $1<q \le p<\infty$.
Then we have
\begin{align*}
\sum_{\varepsilon \in E}
\left\|
\sum_{j=-\infty}^\infty
\left(
\sum_{Q \in {\mathcal D}_j}
\langle f,\chi_Q \rangle 
\cdot
\langle a,h^{\varepsilon}_Q \rangle
h^{\varepsilon}_Q
\right)
\right\|_{{\mathcal M}^p_q}
\lesssim
\|a\|_{{\rm BMO}_{{\rm dyadic}}}
\|f\|_{{\mathcal M}^p_q}.
\end{align*}
\end{theorem}

One formally defines
\begin{equation}\label{eq:9}
I_{\alpha,{\rm dyadic}}f(x)
:=
\sum_{\varepsilon \in E}
\sum_{j=-\infty}^\infty
\sum_{Q \in {\mathcal D}}
|Q|^{\frac{\alpha}{n}}\langle f,h^\varepsilon_Q \rangle h^\varepsilon_Q(x).
\end{equation}
We can justify the definition of $I_{\alpha,{\rm dyadic}}$.
In particular, we can also justify the convergence of the sum
$(\ref{eq:9})$ in the next theorem.
\begin{theorem}
\label{thm3}
Let $0<\alpha<n, \, 1<q \le p<\infty, \, 1<t \le s<\infty$.
Assume 
\begin{equation}
\frac{1}{s}=\frac{1}{p}-\frac{\alpha}{n}, \,
\frac{t}{s}=\frac{q}{p}.
\end{equation}
Then, for every $f \in {\mathcal M}^p_q$,
\begin{align}
\label{eq:100106-1}
I_{\alpha,{\rm dyadic}}f(x)
=\lim_{M \to \infty}
\left(
\sum_{\varepsilon \in E}
\sum_{j=-M}^M \sum_{Q \in {\mathcal D}_j}|Q|^{\frac{\alpha}{n}}
\langle f,h^\varepsilon_Q \rangle h^\varepsilon_Q(x)\right)
\end{align}
converges for almost every $x \in {\mathbb R}^n$
and we have
\begin{equation}
\|I_{\alpha,{\rm dyadic}}f\|_{{\mathcal M}^s_t}
\lesssim
\|f\|_{{\mathcal M}^p_q}.
\end{equation}
\end{theorem}

\begin{theorem}
\label{thm4}
Let $0<\alpha<n, \, 1<q \le p<\infty, \, 1<t \le s<\infty$
and $a \in {\rm BMO}_{{\rm dyadic}}$.
Assume 
\begin{equation}
\frac{1}{s}=\frac{1}{p}-\frac{\alpha}{n}, \,
\frac{t}{s}=\frac{q}{p}.
\end{equation}
Then, for every $f \in {\mathcal M}^p_q$,
the limit
\begin{align*}
[a,I_{\alpha,{\rm dyadic}}]f(x)
&=
a(x)I_{\alpha,{\rm dyadic}}f(x)-I_{\alpha,{\rm dyadic}}[a \cdot f](x)\\
&:=
\lim_{M \to \infty}
\sum_{\varepsilon \in E}
\sum_{j=-M}^M
\sum_{Q \in {\mathcal D}_j}
\langle a \cdot I_{\alpha,{\rm dyadic}}f-I_{\alpha,{\rm dyadic}}[a \cdot f],
h^{\varepsilon}_Q \rangle h^{\varepsilon}_Q(x)
\end{align*}
exists in the topology of $L^{q,{\rm loc}}$
and we have
\begin{equation}
\|\,[a,I_{\alpha,{\rm dyadic}}]f\,\|_{{\mathcal M}^s_t}
\lesssim
\|a\|_{{\rm BMO}_{{\rm dyadic}}}
\|f\|_{{\mathcal M}^p_q}.
\end{equation}
\end{theorem}

Next, we prove
that the operator norm is characterized
by the dyadic BMO norm.
\begin{theorem}
\label{thm5}
Let $a \in {\rm BMO}_{{\rm dyadic}}$.
Suppose that we are given parameters
$p,q,s,t,\alpha$ satisfying
\[
1<q \le p<\infty, \quad
1<t \le s<\infty, \quad
0<\alpha<n
\]
and 
\[
\frac{p}{q}=\frac{t}{s}, \quad
\frac{1}{s}=\frac{1}{p}-\frac{\alpha}{n}.
\]
Then we have
\begin{align*}
\|\,[a,I_{\alpha,{\rm dyadic}}]\,\|_{B({\mathcal M}^p_q,{\mathcal M}^s_t)}
\sim
\|a\|_{{\rm BMO}_{{\rm dyadic}}}.
\end{align*}
\end{theorem}
Needless to say, it is significant
to prove that 
\begin{align*}
\|\,[a,I_{\alpha,{\rm dyadic}}]\,\|_{B({\mathcal M}^p_q,{\mathcal M}^s_t)}
\gtrsim
\|a\|_{{\rm BMO}_{{\rm dyadic}}}
\end{align*}
in view of Theorem \ref{thm4}.
In the usual setting of $p=q$ and $s=t$,
Theorem \ref{thm4} is known
as the result due to S.~Chanillo
\cite{Chanillo-1982-Indiana}.

All the results above carry over to predual spaces.
Recall that the predual space ${\mathcal H}^p_q$
of the Morrey space ${\mathcal M}^{p'}_{q'}$
is given as follows:
Let $1<p \le q<\infty$.
\begin{enumerate}
\item[$(i)$]
A function $A \in L^q$ is said to be a $(p,q)$-block,
if there exists a dyadic cube $Q$ such that
$\|A\|_{L^q} \le |Q|^{\frac{1}{q}-\frac{1}{p}}$
and that $A$ is supported on $Q$.
\item[$(ii)$]
The predual space ${\mathcal H}^p_q$ is given by
\begin{equation}
{\mathcal H}^p_q:=
\left\{
\sum_{j=1}^\infty \lambda_j a_j \, : \,
\sum_{j=1}^\infty |\lambda_j|<\infty
\mbox{ and each }
a_j \mbox{ is a }(p,q)\mbox{-block}
\right\}
\end{equation}
and the norm is given by
\begin{equation}
\|f\|_{{\mathcal H}^p_q}
:=
\inf\left\{
\sum_{j=1}^\infty |\lambda_j| \, : \,
f=\sum_{j=1}^\infty \lambda_j a_j 
\mbox{ and each }
a_j \mbox{ is a }(p,q)\mbox{-block}
\right\}
\end{equation}
for $f \in {\mathcal H}^p_q$.
\end{enumerate}

A well-known fact is that
the dual of ${\mathcal H}^{p'}_{q'}$ 
is ${\mathcal M}^p_q$ (see \cite{Zo}).
Therefore, it seems easy to prove this theorem
by duality.
\begin{theorem}\label{thm6}
Let $0<\alpha<n$, \, $1<r \le r_0<\infty$ and $1<p \le p_0<\infty$.
Assume in addition 
$$
\frac{1}{r_0}=\frac{1}{p_0}-\frac{\alpha}{n}, \quad
\frac{r}{r_0}=\frac{p}{p_0}.
$$
\begin{enumerate}
\item[$(i)$]
The fractional integral operator $I_{\alpha,{\rm dyadic}}$,
which is originally defined on $L^{r_0'}$,
is bounded from ${\mathcal H}^{r_0'}_{r'}$ 
to ${\mathcal H}^{p_0'}_{p'}$.
That is,
\[
\|I_{\alpha,{\rm dyadic}}f\|_{{\mathcal H}^{p_0'}_{p'}}
\le C
\|f\|_{{\mathcal H}^{r_0'}_{r'}}
\]
for all $f \in {\mathcal H}^{r_0'}_{r'}$
\item[$(ii)$]
The commutator $[a,I_{\alpha,{\rm dyadic}}]$,
which is originally defined on $L^{r_0'}$,
is bounded from ${\mathcal H}^{r_0'}_{r'}$ 
to ${\mathcal H}^{p_0'}_{p'}$.
\end{enumerate}
\end{theorem}
Actually, we invoke dualtiy to prove this theorem.
However, we need to pay attention
to perform duality argument.
Here is a \lq \lq wrong" proof for $I_{\alpha,{\rm dyadic}}$.
The same can be said for $[a,I_{\alpha,{\rm dyadic}}]$
or $I_\alpha$.
\begin{proof}[{\bf Wrong proof of Thoerem \ref{thm6}}]
By duality argument,
we have
\[
\|I_{\alpha,{\rm dyadic}}f\|_{{\mathcal H}^{p_0'}_{p'}}
=
\sup
\left\{
\left|\int_{{\mathbb R}^n}I_{\alpha,{\rm dyadic}}f(x)h(x)\,dx\right|
\,:\,h \in {\mathcal M}^{p_0}_p, \, \|h\|_{{\mathcal M}^{p_0}_p}=1\right\}.
\]
In view of the definition of $I_{\alpha,{\rm dyadic}}$,
we have
\[
\int_{{\mathbb R}^n}I_{\alpha,{\rm dyadic}}f(x)h(x)\,dx
=
\int_{{\mathbb R}^n}f(x)I_{\alpha,{\rm dyadic}}h(x)\,dx.
\]
If we invoke the boundedness of $I_{\alpha,{\rm dyadic}}$
obtained in Theorem \ref{thm3}
and we denote by $
\|I_{\alpha,{\rm dyadic}}\|_{B({\mathcal M}^{p_0}_{p},{\mathcal M}^{r_0}_{r})}$
the operator norm,
then we have
\[
\|I_{\alpha,{\rm dyadic}}f\|_{{\mathcal H}^{p_0'}_{p'}}
\le 
\|I_{\alpha,{\rm dyadic}}\|_{B({\mathcal M}^{p_0}_{p},{\mathcal M}^{r_0}_{r})}
\|f\|_{{\mathcal H}^{r_0'}_{r'}}.
\]
The proof is now complete.
\end{proof}
Here is some gap in the proof:
There is no guarantee for $I_{\alpha,{\rm dyadic}}f$
to be a member of ${\mathcal H}^{p_0'}_{p'}$.
So to overcome this trouble,
we need to take full advantage of the dyadic fractional integral  operator
$I_{\alpha,{\rm dyadic}}$:
$I_{\alpha,{\rm dyadic}} h^\varepsilon_R
=|R|^{\frac{\alpha}{n}}h^\varepsilon_R$.

Next, we investigate the compactness of the commutator
$[a,I_{\alpha,{\rm dyadic}}]$.
To this end we define
${\rm VMO}_{{\rm dyadic}}$ as the closure of 
${\rm Span}
(\{h^\varepsilon_Q\}_{\varepsilon \in E, \, Q \in {\mathcal D}})$,
where ${\rm Span}(A)$ denotes a linear subspace
generated by a set $A$.

\begin{theorem}
\label{thm7}
Let $0<\alpha<n$, \,$1<q \le p<\infty$ and $1<t \le s<\infty$.
Assume 
\begin{equation}
\frac{1}{s}=\frac{1}{p}-\frac{\alpha}{n}, \quad
\frac{t}{s}=\frac{q}{p}.
\end{equation}
Then $a \in {\rm BMO}_{{\rm dyadic}}$ generates a compact commutator 
$[a,I_{\alpha,{\rm dyadic}}]:{\mathcal M}^p_q \to {\mathcal M}^s_t$
if and only if 
$a \in {\rm VMO}_{{\rm dyadic}}$.
\end{theorem}
We remark that the \lq\lq if" part of Theorem \ref{thm7} 
is investigated in
\cite{Sawano-Shirai-2008-GMJ}.

All the theorems above are proved in Section 3
after collecting some auxiliary facts
in Section \ref{section:preliminaries}.

\section{Preliminaries}
\label{section:preliminaries}

Here we collect some preliminary facts.
For the proof of Proposition \ref{prop:2.1}
we refer to \cite[Chapter 2]{HW}.
\begin{proposition}
\label{prop:2.1}
Let $1<q<\infty$.
\begin{enumerate}
\item[$(i)$]
For $f \in L^q$, 
the following equivalence holds{\rm:}
\begin{equation}
\label{eq:100104-1}
\|f\|_{L^q}
\sim
\sum_{\varepsilon \in E}
\left\|
\left(\sum_{j=-\infty}^\infty\left|\sum_{Q \in {\mathcal D}_j}
\langle f,h^{\varepsilon}_Q \rangle h^{\varepsilon}_Q
\right|^2\right)^{1/2}
\right\|_{L^q}.
\end{equation}
\item[$(ii)$]
For $f \in L^q$ and $k \in {\mathbb Z}$,
the following equivalence holds{\rm:}
\begin{equation}
\label{eq:100104-2}
\|f\|_{L^q}
\sim
\sum_{\varepsilon \in E}
\left\|
\left(\sum_{j=k}^\infty
\left|\sum_{Q \in {\mathcal D}_j}
\langle f,h^{\varepsilon}_Q \rangle h^{\varepsilon}_Q\right|^2\right)^{1/2}
\right\|_{L^q}
+
\left\|\sum_{Q \in {\mathcal D}_k}
\langle f,\chi_Q \rangle \frac{\chi_Q}{|Q|}\right\|_{L^q}.
\end{equation}
\item[$(iii)$]
For $f \in L^{q,{\rm loc}}$ and $R \in {\mathcal D}$,
the following equivalence holds{\rm:}
\begin{equation}
\label{eq:100104-3}
\|f-m_R(f)\|_{L^q(R)}
\sim
\sum_{\varepsilon \in E}
\left\|
\left(\sum_{j=-\log_2\ell(R)}^{\infty}
\left|\sum_{Q \in {\mathcal D}_j, \, Q \subset R}
\langle f,h^{\varepsilon}_Q \rangle h^{\varepsilon}_Q\right|^2\right)^{1/2}
\right\|_{L^q}.
\end{equation}
\end{enumerate}
Here the implicit constant in {\rm (\ref{eq:100104-2})}
does not depend on $k$.
\end{proposition}
A counterpart of Proposition \ref{prop:2.1}
for Herz spaces was proved in \cite{Izuki-Sawano-2009-JMAA}.
So, it seems possible to extend the results
to these spaces.

This is the only propositions
whose proof we omit in the present paper.

When $n=1$,
the next proposition is \cite[Theorem 2.6.]{Lacey-2005-Hokkaido}.
\begin{proposition}\label{prop:2.2}
Let $a \in {\rm BMO}_{{\rm dyadic}}$ and $1<q<\infty$.
Then
the following is an equivalent norm 
of $\|a\|_{{\rm BMO}_{\rm dyadic}}${\rm:}
\begin{equation*}
\sup
\left\{
\sum_{\varepsilon \in E}
\left\|
\sum_{j=-\infty}^\infty
\left(
\sum_{Q \in {\mathcal D}_j}
\langle f,\chi_Q \rangle 
\cdot
\langle a,h^{\varepsilon}_Q \rangle
\frac{h^{\varepsilon}_Q}{|Q|}
\right)
\right\|_{L^q}
\,:\,f \in L^q, \, \|f\|_{L^q}=1
\right\}.
\end{equation*}
\end{proposition}
This theorem is motivated by
the results due to Coifman and Meyer.
(See \cite{Coifman-Meyer-1975-TAMS,Coifman-Meyer-1978-Grenoble}.)
\begin{proof}
This is somehow well known \cite[Theorem 2.6.]{Lacey-2005-Hokkaido}.
The proof of
\[
\sup
\left\{
\sum_{\varepsilon \in E}
\left\|
\sum_{j=-\infty}^\infty
\left(
\sum_{Q \in {\mathcal D}_j}
\langle f,\chi_Q \rangle 
\cdot
\langle a,h^{\varepsilon}_Q \rangle
\frac{h^{\varepsilon}_Q}{|Q|}
\right)
\right\|_{L^q}
\,:\,f \in L^q, \, \|f\|_{L^q}=1
\right\}
\gtrsim 
\|a\|_{{\rm BMO}_{\rm dyadic}}
\]
can be proved by the cube testing and Proposition \ref{prop:2.1}.
For the reverse inequality,
we use an argument of $T1$-type
as well as the Carleson embedding theorem when $p=2$.
The situation resembles that 
in \cite[p.302 (64)]{St}.
\end{proof}

Here and below, for $k \in {\mathbb N} \cup \{0\}$
and $R \in {\mathcal D}$, 
we write $R_{+k}$
for the unique dyadic cube containing $R$ and 
of volume $2^{kn}|R|$.
\begin{equation}
\label{eq:dyadic-k}
|R_{+k}|=2^{kn}|R|, \, R \subset R_{+k}, \, R_{+k} \in {\mathcal D}.
\end{equation}

Finally before we prove Theorems \ref{thm1}--\ref{thm7},
we shall obtain a counterpart of Proposition \ref{prop:2.2}
for Morrey spaces.

\begin{proposition}
Let $1<q \le p<\infty$.
Then we have
\begin{align}\label{eq:22}
\sup_f
\sum_{\varepsilon \in E}
\left\|
\sum_{j=-\infty}^\infty
\left(
\sum_{Q \in {\mathcal D}_j}
\langle f,\chi_Q \rangle 
\cdot
\langle a,h^{\varepsilon}_Q \rangle
\frac{h^{\varepsilon}_Q}{|Q|}
\right)
\right\|_{{\mathcal M}^p_q}
\sim
\|a\|_{{\rm BMO}_{{\rm dyadic}}},
\end{align}
where the supremum is taken over all
$f \in{\mathcal M}^p_q$ such that $\|f\|_{{\mathcal M}^p_q}=1$.
\end{proposition}

\begin{proof}
Observe that by Proposition \ref{prop:2.1}
(see $(\ref{eq:100104-3})$) we have
\[
\|a\|_{{\rm BMO}_{{\rm dyadic}}}
\sim
\sup_{Q \in {\mathcal D}}
\left(
\frac{1}{|Q|}\int_Q 
\sum_{j=-\infty}^\infty
\left|\sum_{R \in {\mathcal D}_j, \, R \subset Q}
\langle a,h^\varepsilon_R \rangle h^\varepsilon_R(x)
\right|^{\frac{q}{2}}\,dx
\right)^{\frac{1}{q}}.
\]
Consequently
the inequality $\gtrsim$ in (\ref{eq:22}) follows
by considering $f=|Q|^{-1/p-1/2}h^\varepsilon_Q$
for a cube $Q$.

Let us prove 
the inequality $\lesssim$ in (\ref{eq:22}).
Let $S$ be a fixed dyadic cube.
Then we need to show that
\[
|S|^{\frac{1}{p}-\frac{1}{q}}
\left(
\int_S
\left|
\sum_{j=-\infty}^\infty
\left(
\sum_{Q \in {\mathcal D}_j}
\langle f,\chi_Q \rangle 
\cdot
\langle a,h^{\varepsilon}_Q \rangle
\frac{h^{\varepsilon}_Q(x)}{|Q|}
\right)
\right|^q\,dx
\right)^{\frac{1}{q}}
\lesssim
\|a\|_{{\rm BMO}_{{\rm dyadic}}}
\]
for all $f \in {\mathcal M}^p_q$ with norm $1$.

By Proposition \ref{prop:2.2},
we have
\begin{align*}
\lefteqn{
|S|^{\frac{1}{p}-\frac{1}{q}}
\left(
\int_S
\left|
\sum_{j=-\infty}^\infty
\left(
\sum_{Q \in {\mathcal D}_j}
\langle \chi_S f,\chi_Q \rangle 
\cdot
\langle a,h^{\varepsilon}_Q \rangle
\frac{h^{\varepsilon}_Q(x)}{|Q|}
\right)
\right|^q\,dx
\right)^{\frac{1}{q}}
}\\
&\le
|S|^{\frac{1}{p}-\frac{1}{q}}
\left(
\int_{{\mathbb R}^n}
\left|
\sum_{j=-\infty}^\infty
\left(
\sum_{Q \in {\mathcal D}_j}
\langle \chi_S f,\chi_Q \rangle 
\cdot
\langle a,h^{\varepsilon}_Q \rangle
\frac{h^{\varepsilon}_Q(x)}{|Q|}
\right)
\right|^q\,dx
\right)^{\frac{1}{q}}\\
&\lesssim
\|a\|_{{\rm BMO}_{{\rm dyadic}}}
|S|^{\frac{1}{p}-\frac{1}{q}}\|\chi_S f\|_{L^q}.
\end{align*}
By the definition of the Morrey norm $\|f\|_{{\mathcal M}^p_q}$
we have
\begin{align*}
|S|^{\frac{1}{p}-\frac{1}{q}}
\left(
\int_S
\left|
\sum_{j=-\infty}^\infty
\left(
\sum_{Q \in {\mathcal D}_j}
\langle \chi_S f,\chi_Q \rangle 
\cdot
\langle a,h^{\varepsilon}_Q \rangle
\frac{h^{\varepsilon}_Q(x)}{|Q|}
\right)
\right|^q\,dx
\right)^{\frac{1}{q}}
&\lesssim
\|a\|_{{\rm BMO}_{{\rm dyadic}}}.
\end{align*}
Meanwhile,
a geometric observation shows that
\begin{align}
\lefteqn{
|S|^{\frac{1}{p}-\frac{1}{q}}
\left(
\int_S
\left|
\sum_{j=-\infty}^\infty
\left(
\sum_{Q \in {\mathcal D}_j}
\langle \chi_{{\mathbb R}^n \setminus S} f,\chi_Q \rangle 
\cdot
\langle a,h^{\varepsilon}_Q \rangle
\frac{h^{\varepsilon}_Q(x)}{|Q|}
\right)
\right|^q\,dx
\right)^{\frac{1}{q}}
}\nonumber\\
\label{eq:2500}
&=|S|^{\frac{1}{p}-\frac{1}{q}}
\left(
\int_S
\left|
\sum_{k=0}^\infty
\left(
\langle \chi_{{\mathbb R}^n \setminus S} f,\chi_{S_{+k}} \rangle 
\cdot
\langle a,h^{\varepsilon}_{S_{+k}} \rangle
\frac{h^{\varepsilon}_{S_{+k}}(x)}{|{S_{+k}}|}
\right)
\right|^q\,dx
\right)^{\frac{1}{q}},
\end{align}
where $S_{+k}$ is given by (\ref{eq:dyadic-k})
with $R$ replaced with $S$.
With the definition 
of the Morrey norm (\ref{eq:Morrey}),
a crude estimate
\[
|\langle a,h^{\varepsilon}_Q \rangle| \le 
C|Q|^{1/2}\|a\|_{\rm BMO}
\]
and this observation $(\ref{eq:2500})$
in mind,
we obtain
\[
|S|^{\frac{1}{p}-\frac{1}{q}}
\left(
\int_S
\left|
\sum_{j=-\infty}^\infty
\left(
\sum_{Q \in {\mathcal D}_j}
\langle \chi_{{\mathbb R}^n \setminus S} f,\chi_Q \rangle 
\cdot
\langle a,h^{\varepsilon}_Q \rangle
\frac{h^{\varepsilon}_Q(x)}{|Q|}
\right)
\right|^q\,dx
\right)^{\frac{1}{q}}
\lesssim
\|a\|_{{\rm BMO}_{{\rm dyadic}}}.
\]
The inequality $\lesssim$ in (\ref{eq:22}) is proved
and 
the proof is therefore complete.
\end{proof}

\section{Proof of Theorems}

\subsection{Proof of Theorem \ref{thm1}}
\label{subsection:thm1}

We shall prove an auxiliary inequality
which is interesting of its own right.
\begin{lemma}\label{lem:1}
Let $1 \le q \le p<\infty$.
Let $f \in {\mathcal M}^p_q$.
Then we have equivalence
\begin{equation}
\|f\|_{{\mathcal M}^p_q}
\sim
\sum_{\varepsilon \in E}
\left\|\left(\sum_{j=-\infty}^\infty\left|\sum_{Q \in {\mathcal D_j}}
\langle f,h_Q^\varepsilon \rangle h_Q^{\varepsilon}\right|^2
\right)^{\frac12}\right\|_{{\mathcal M}^p_q}
+
\|f\|_{{\mathcal M}^p_1}.
\end{equation}
\end{lemma}

\begin{proof}
Let $R \in {\mathcal D}$ be fixed throughout the proof.

By virtue of a crude estimate and the H\"{o}lder inequality
\begin{equation}
|\langle f,h_{R_{+m}}^\varepsilon \rangle 
h_{R_{+m}}^{\varepsilon}|
\le
|R_{+m}|^{\frac1p-1}\|f\|_{{\mathcal M}^p_1}
\end{equation}
we obtain
\begin{equation}
\label{eq:100107-1}
|R|^{\frac1p-\frac1q}
\left\{
\int_R\left(\sum_{j=-\infty}^\infty
\left|\sum_{Q \in {\mathcal D_j}, \, Q \supset R}
\langle f,h_Q^\varepsilon 
\rangle h_Q^{\varepsilon}\right|^2
\right)^{\frac{q}2}
\right\}^{\frac1q}
\lesssim \|f\|_{{\mathcal M}^p_1}\le\|f\|_{{\mathcal M}^p_q}.
\end{equation}

Keeping in mind $(\ref{eq:100107-1})$,
let us first prove that
\begin{align}
\label{eq:13}
\|f\|_{{\mathcal M}^p_q}
\gtrsim
\sum_{\varepsilon \in E}
\left\|\left(\sum_{j=-\infty}^\infty\left|\sum_{Q \in {\mathcal D_j}}
\langle f,h_Q^\varepsilon \rangle h_Q^{\varepsilon}\right|^2
\right)^{\frac12}\right\|_{{\mathcal M}^p_q}
+
\|f\|_{{\mathcal M}^p_1}.
\end{align}
By Proposition \ref{prop:2.1} $(i)$ and (\ref{eq:100107-1})
we have
\begin{align*}
\lefteqn{
|R|^{\frac1p-\frac1q}
\left\{
\int_R\left(\sum_{j=-\infty}^\infty\left|\sum_{Q \in {\mathcal D_j}}
\langle f,h_Q^\varepsilon \rangle h_Q^{\varepsilon}\right|^2
\right)^{\frac{q}2}
\right\}^{\frac1q}
}\\
&\lesssim
|R|^{\frac1p-\frac1q}
\left(
\int_R|f(x)|^q\,dx
\right)^{\frac1q}
+
|R|^{\frac1p-\frac1q}
\left\{
\int_R\left(\sum_{j=-\infty}^\infty\left|\sum_{Q \in {\mathcal D_j}}
\langle \chi_{{\mathbb R}^n \setminus R}f,
h_Q^\varepsilon \rangle h_Q^{\varepsilon}\right|^2
\right)^{\frac{q}2}
\right\}^{\frac1q}\\
&\lesssim
\|f\|_{{\mathcal M}^p_q}.
\end{align*}
Thus, (\ref{eq:13}) is established.

Now let us prove the converse inequality
of (\ref{eq:13}).
First by the triangle inequality 
and the definition of the Morrey norm $(\ref{eq:Morrey})$,
we have
\[
|R|^{\frac1p-\frac1q}
\left(
\int_R|f(x)|^q\,dx
\right)^{\frac1q}
\le
|R|^{\frac1p-\frac1q}
\left(
\int_R|f(x)-m_R(f)|^q\,dx
\right)^{\frac1q}
+\|f\|_{{\mathcal M}^p_1}.
\]
By Proposition \ref{prop:2.1} $(iii)$
we obtain
\begin{align*}
\lefteqn{
|R|^{\frac1p-\frac1q}
\left(
\int_R|f(x)|^q\,dx
\right)^{\frac1q}
}\\
&\lesssim
|R|^{\frac1p-\frac1q}
\left(
\int_R
\left(\sum_{j=-\infty}^\infty
\left| \sum_{Q \in {\mathcal D}_j}
\langle h^{\varepsilon}_Q f \rangle
h^{\varepsilon}_Q(x)-m_R
\left[\sum_{Q \in {\mathcal D}_j}
\langle h^{\varepsilon}_Q f \rangle
h^{\varepsilon}_Q\right]
\right|^2\right)^{\frac{q}2}\,dx
\right)^{\frac1q}
+\|f\|_{{\mathcal M}^p_1}\\
&=
|R|^{\frac1p-\frac1q}
\left(
\int_R
\left(\sum_{j=-\infty}^\infty
\left| \sum_{Q \in {\mathcal D}_j, \, Q\subset R}
\langle h^{\varepsilon}_Q f \rangle
h^{\varepsilon}_Q(x)
\right|^2\right)^{\frac{q}2}\,dx
\right)^{\frac1q}
+\|f\|_{{\mathcal M}^p_1}.
\end{align*}
If we use (\ref{eq:100107-1}) again,
then we have
\begin{align*}
\lefteqn{
|R|^{\frac1p-\frac1q}
\left(
\int_R\left|\sum_{j=-\infty}^\infty \sum_{Q \in {\mathcal D}_j, \,
Q \subset R}
\langle h^{\varepsilon}_Q f \rangle
h^{\varepsilon}_Q(x)\right|^q\,dx
\right)^{\frac1q}
}\\
&\lesssim
|R|^{\frac1p-\frac1q}
\left(
\int_R\left|\sum_{j=-\infty}^\infty \sum_{Q \in {\mathcal D}_j}
\langle h^{\varepsilon}_Q f \rangle
h^{\varepsilon}_Q(x)\right|^q\,dx
\right)^{\frac1q}
+\|f\|_{{\mathcal M}^p_1}
\end{align*}
Thus, the proof of Lemma \ref{lem:1} is complete.
\end{proof}
Let us now prove Theorem \ref{thm1}.
In view of Lemma \ref{lem:1},
for the proof of $(i)$
it suffices to establish
\begin{equation}\label{eq:27}
|R|^{\frac1p}m_R(|f|)
=
|R|^{\frac1p-1}\int_R|f(x)|\,dx
\lesssim
\sum_{\varepsilon \in E}
\left\|\left(\sum_{j=-\infty}^\infty\left|\sum_{Q \in {\mathcal D_j}}
\langle f,h_Q^\varepsilon \rangle h_Q^{\varepsilon}\right|^2
\right)^{\frac12}\right\|_{{\mathcal M}^p_q}.
\end{equation}
By the triangle inequality,
we have
\begin{align*}
|R|^{\frac1p}m_R(|f|)
&=
\lim_{k \to \infty}
|R|^{\frac1p-1}\int_R|f(x)-m_{R_{+k}}(f)|\,dx\\
&\lesssim
\sum_{k=0}^\infty
2^{-k\left(\frac1p-1\right)}|R_{+k}|^{\frac1p-1}
\int_{R_{+k}}|f(x)-m_{R_{+k}}(f)|\,dx
\end{align*}
We calculate, by using Proposition \ref{prop:2.1}
and the fact that $p>1$,
\begin{align*}
|R|^{\frac1p}m_R(|f|)
&\lesssim
\sum_{\varepsilon \in E}
\sum_{k=0}^\infty
2^{-k\left(\frac1p-1\right)}
|R_{+k}|^{\frac1p-\frac1q}
\left(
\int_{R_{+k}}\left|\sum_{j=-\infty}^\infty \sum_{Q \in {\mathcal D}_j}
\langle h^{\varepsilon}_Q f \rangle
h^{\varepsilon}_Q(x)\right|^q\,dx
\right)^{\frac1q}\\
&\lesssim
\sum_{\varepsilon \in E}
\left\|\left(\sum_{j=-\infty}^\infty\left|\sum_{Q \in {\mathcal D_j}}
\langle f,h_Q^\varepsilon \rangle h_Q^{\varepsilon}\right|^2
\right)^{\frac12}\right\|_{{\mathcal M}^p_q}.
\end{align*}
As a conseqeunce (\ref{eq:27}) is proved.

Therefore, the proof of $(i)$ is complete.

For the proof of $(ii)$
we fix a compact set $K \subset {\mathbb R}^n$
and prove that
\begin{equation}
\label{eq:100107-5}
\lim_{M \to \infty}
\sum_{\varepsilon \in E}
\left\|\sum_{j=-\infty}^\infty
\sum_{Q \in {\mathcal D}_j}
\langle f,h^{\varepsilon}_Q \rangle h^{\varepsilon}_Q
-\sum_{j=-M}^M
\sum_{Q \in {\mathcal D}_j}
\langle f,h^{\varepsilon}_Q \rangle h^{\varepsilon}_Q
\right\|_{L^q(K)}=0.
\end{equation}
However, since $K$ can be covered by $3^n$ dyadic cubes
of the same size,
we have only to prove (\ref{eq:100107-5})
with $K$ replaced by a dyadic cube $R \in {\mathcal D}_k$,
where $k \in {\mathbb Z}$ is a fixed integer.
Let us denote
\begin{equation}
f^\varepsilon_j:=\sum_{Q \in {\mathcal D}_j}
\langle f,h^{\varepsilon}_Q \rangle h^{\varepsilon}_Q
\end{equation}
for $\varepsilon \in E$ and $j \in {\mathbb Z}$.
If $x \in R$ and $M \ge 1+|k|$,
then we have
\begin{align*}
\sum_{\varepsilon \in E}
\left|\sum_{j=-\infty}^\infty
f^\varepsilon_j(x)
-\sum_{j=-M}^M
f^\varepsilon_j(x)
\right|
&=
\sum_{\varepsilon \in E}
\left|\sum_{j=M+1}^\infty
f^\varepsilon_j(x)
+
\sum_{j=-\infty}^{-M-1}
f^\varepsilon_j(x)
\right|\\
&\le
\sum_{\varepsilon \in E}
\left|\sum_{j=M+1}^\infty
\sum_{\substack{Q \in {\mathcal D}_j \\ Q \subset R}}
\langle f,h^{\varepsilon}_Q \rangle h^{\varepsilon}_Q(x)
\right|
+
\sum_{\varepsilon \in E}
\left|
\sum_{j=-\infty}^{-M-1}
f^\varepsilon_j(x)
\right|.
\end{align*}
Recall that $R \in {\mathcal D}_k$.
Consequently,
we can write
\begin{equation}\label{eq:330}
f^\varepsilon_j(x)
=\sum_{Q \in {\mathcal D}_j}
\langle f,h^{\varepsilon}_Q \rangle h^{\varepsilon}_Q(x)
=\langle f,h^{\varepsilon}_{R_{+(-j+k)}} \rangle 
h^{\varepsilon}_{R_{+(-j+k)}}(x)
\quad (x \in R_k),
\end{equation}
if $j$ is negative enough, that is, $j \le -M-1<-|k|$.
If we use $(\ref{eq:330})$,
then we obtain
\begin{align*}
\lefteqn{
\sum_{\varepsilon \in E}
\left|\sum_{j=-\infty}^\infty
f^\varepsilon_j(x)
-\sum_{j=-M}^M
f^\varepsilon_j(x)
\right|
}\\
&\le
\sum_{\varepsilon \in E}
\left|\sum_{j=M+1}^\infty
\sum_{Q \in {\mathcal D}_j, \, Q \subset R}
\langle f,h^{\varepsilon}_Q \rangle h^{\varepsilon}_Q(x)
\right|
+
\sum_{\varepsilon \in E}
\sum_{m=M+k+1}^\infty 
|\langle f,h^{\varepsilon}_{R_{+m}} \rangle 
h^{\varepsilon}_{R_{+m}}(x)|.
\end{align*}
Thus,
by the triangle inequality, we have
\begin{align*}
\lefteqn{
\left\|\sum_{j=-\infty}^\infty
f^\varepsilon_j
-\sum_{j=-M}^M
f^\varepsilon_j
\right\|_{L^q(R)}
}\\
&\le
\left\|\sum_{j=M+1}^\infty
\sum_{Q \in {\mathcal D}_j, \, Q \subset R}
\langle f,h^{\varepsilon}_Q \rangle h^{\varepsilon}_Q
\right\|_{L^q(R)}
+
\sum_{m=M+k+1}^\infty 
\left\|\langle f,h^{\varepsilon}_{R_{+m}} \rangle 
h^{\varepsilon}_{R_{+m}}
\right\|_{L^q(R)}.
\end{align*}
A geometric observation shows that
\begin{align*}
\lefteqn{
\left\|\langle f,h^{\varepsilon}_{R_{+m}} \rangle 
h^{\varepsilon}_{R_{+m}}
\right\|_{L^q(R)}
}\\
&=2^{-mn/q}
\left\|\langle f,h^{\varepsilon}_{R_{+m}} \rangle 
h^{\varepsilon}_{R_{+m}}
\right\|_{L^q(R_{+m})}\\
&=2^{-mn/p+kn(1/p-1/q)}
|R_{+m}|^{1/p-1/q}
\left\|\langle f,h^{\varepsilon}_{R_{+m}} \rangle 
h^{\varepsilon}_{R_{+m}}
\right\|_{L^q(R_{+m})}.
\end{align*}
If we use this equality,
then we have
\begin{align*}
\lefteqn{
\left\|\sum_{j=-\infty}^\infty
f^\varepsilon_j
-\sum_{j=-M}^M
f^\varepsilon_j
\right\|_{L^q(R)}
}\\
&=
\left\|\sum_{j=M+1}^\infty
\sum_{Q \in {\mathcal D}_j, \, Q \subset R}
\langle f,h^{\varepsilon}_Q \rangle h^{\varepsilon}_Q
\right\|_{L^q(R)}
+
\sum_{m=M+k+1}^\infty 
2^{-mn/q}
\left\|\langle f,h^{\varepsilon}_{R_{+m}} \rangle 
h^{\varepsilon}_{R_{+m}}
\right\|_{L^q(R_{+m})}\\
&\lesssim
\left\|\left(\sum_{j=M+1}^\infty|f^\varepsilon_j|^2
\right)^{\frac12}\right\|_{L^q(R)}
+2^{-nM/p-nk/q}
\left\|\left(
\sum_{j=-\infty}^\infty|f^\varepsilon_j|^2
\right)^{\frac12}\right\|_{{\mathcal M}^p_q}.
\end{align*}
Thus, we obtain (\ref{eq:100107-5}),
which shows that (\ref{eq:thm1-3}) holds
in the topology of $L^{q,{\rm loc}}$.
As a consequence Theorem \ref{thm1} is proved completely.

\begin{remark}
It may be interesting to compare 
Theorem \ref{thm1} and Lemma \ref{lem:1}
with the result in \cite[Theorem 1.3]{SaTa2}.
In \cite[Theorem 1.3]{SaTa2},
we have proved that
\[
\|f\|_{{\mathcal M}^p_q}
\sim
\|M^\sharp f\|_{{\mathcal M}^p_q}
+
\|f\|_{{\mathcal M}^p_1}\,(1<q \le p<\infty).
\]
Here $M^\sharp$ denotes the sharp maximal operator
due to Fefferman, Stein and Stromberg.
\end{remark}

\subsection{Proof of Theorem \ref{thm2}}
\label{subsection:thm2}

Let $\varepsilon \in E$ be fixed.
We also take a dyadic cube $R$.
Then it suffices from Theorem \ref{thm1}
\begin{align*}
{\rm I}
&:=
|R|^{\frac1p-\frac1q}
\left\{
\int_R
\left(
\sum_{j=-\log_2\ell(R)}^\infty
\left|
\sum_{Q \in {\mathcal D}_j, \, Q \subset R}
\langle f,\chi_Q \rangle 
\cdot
\langle a,h^{\varepsilon}_Q \rangle
h^{\varepsilon}_Q(x)
\right|^2
\right)^{\frac{q}{2}}\,dx
\right\}^{\frac1q}
\end{align*}
by $C\|a\|_{{\rm BMO}_{dyadic}}\|f\|_{{\mathcal M}^p_q}$
with constants independent of $R$, $a$ and $f$.
By using Proposition \ref{prop:2.2}
we obtain
\begin{align*}
{\rm I}
&=
|R|^{\frac1p-\frac1q}
\left\{
\int_R
\left(
\sum_{j=-\log_2\ell(R)}^\infty
\left|
\sum_{Q \in {\mathcal D}_j, \, Q \subset R}
\langle \chi_R f,\chi_Q \rangle 
\cdot
\langle a,h^{\varepsilon}_Q \rangle
h^{\varepsilon}_Q(x)
\right|^2
\right)^{\frac{q}{2}}\,dx
\right\}^{\frac1q}.
\end{align*}
Note 
that $\{Q\,:\,Q \in {\mathcal D}_j\}$ partitions ${\mathbb R}^n$.
Consequently,
we have
\begin{align*}
{\rm I}
&\lesssim \|a\|_{{\rm BMO}_{{\rm dyadic}}}
|R|^{\frac1p-\frac1q}
\left\{
\int_R
\left(
\sum_{j=-\log_2\ell(R)}^\infty
\left|
\sum_{Q \in {\mathcal D}_j, \, Q \subset R}
\langle \chi_R f,\chi_Q \rangle 
|Q|^{\frac12}
h^{\varepsilon}_Q(x)
\right|^2
\right)^{\frac{q}{2}}\,dx
\right\}^{\frac1q}
\end{align*}
from the definition of $\|a\|_{{\rm BMO}_{{\rm dyadic}}}$.
If we use the definition of the Morrey norm (\ref{eq:Morrey})
crudely,
then we have
\begin{align*}
{\rm I}
\lesssim \|a\|_{{\rm BMO}_{{\rm dyadic}}}
|R|^{\frac1p-\frac1q}
\left(\int_R |f(x)|^q\,dx\right)^{\frac1q}
\lesssim \|a\|_{{\rm BMO}_{{\rm dyadic}}}
\|f\|_{{\mathcal M}^p_q}.
\end{align*}
Therefore, since $R$ is arbitrary,
the proof of Theorem \ref{thm2} is complete.

\subsection{Proof of Theorem \ref{thm3}}
\label{subsection:thm3}

We shall make use of the following estimate
in the proof of Theorem \ref{thm4} as well as Theorem \ref{thm3}.
Actually Proposition \ref{prop:1} is a little stronger
than Theorem \ref{thm3}.
\begin{proposition}
\label{prop:1}
Let $0<\alpha<n$, $1<q \le p<\infty$ and $1<t \le s<\infty$.
Assume 
\begin{equation}
\frac{1}{s}=\frac{1}{p}-\frac{\alpha}{n}, \quad
\frac{t}{s}=\frac{q}{p}.
\end{equation}
Then
\begin{equation}
\|I_{\alpha,{\rm dyadic}}f\|_{{\mathcal M}^s_t}
\lesssim
\left\|
\sum_{j=-\infty}^\infty 
\left|
\sum_{Q \in {\mathcal D}_j}
|Q|^{\frac{\alpha}{n}}
\langle f,h^{\varepsilon}_Q \rangle h^{\varepsilon}_Q
\right|\,
\right\|_{{\mathcal M}^s_t}
\lesssim
\|f\|_{{\mathcal M}^p_q}.
\end{equation}
\end{proposition}

\begin{proof}
The proof is simple:
Let $x \in {\mathbb R}^n$ and $j \in {\mathbb Z}$ be fixed
and choose $Q_0 \in {\mathcal D}_j$ so that $x \in Q_0$.
If we use a simple inequality
\begin{equation}
\left|\sum_{Q \in {\mathcal D}_j}
|Q|^{\frac{\alpha}{n}}\langle f,h_Q^{\varepsilon} \rangle 
h_Q^{\varepsilon}(x)\right|
\le
\ell(Q_0)^{\alpha}m_{Q_0}(|f|)
\le
\ell(Q_0)^{\alpha-\frac{n}{p}}\|f\|_{{\mathcal M}^p_q}
=
2^{-j\left(\alpha-\frac{n}{p}\right)}\|f\|_{{\mathcal M}^p_q},
\end{equation}
we have
\begin{align*}
\sum_{j=-\infty}^\infty\left|
\sum_{Q \in {\mathcal D}_j}
|Q|^{\frac{\alpha}{n}}\langle f,h_Q^{\varepsilon} \rangle h_Q^{\varepsilon}
\right|
&\le
\sum_{j=-\infty}^\infty
2^{-j\alpha}
\min\left(\left|\sum_{Q \in {\mathcal D}_j}
\langle f,h_Q^{\varepsilon} \rangle h_Q^{\varepsilon}\right|,
2^{\frac{jn}{p}}\|f\|_{{\mathcal M}^p_q}\right)\\
&\le
\sum_{j=-\infty}^\infty
2^{-j\alpha}
\min\left(\sup_{l \in {\mathbb Z}}\left|\sum_{Q \in {\mathcal D}_l}
\langle f,h_Q^{\varepsilon} \rangle h_Q^{\varepsilon}\right|,
2^{\frac{jn}{p}}\|f\|_{{\mathcal M}^p_q}\right)\\
&\lesssim
\|f\|_{{\mathcal M}^p_q}^{1-\frac{p}{s}}
\sup_{l \in {\mathbb Z}}
\left|\sum_{Q \in {\mathcal D}_l}
\langle f,h_Q^{\varepsilon} \rangle h_Q^{\varepsilon}
\right|^{\frac{p}{s}}.
\end{align*}
If we use this pointwise estimate,
then we obtain the desired estimate.
\end{proof}

\begin{remark}
It may be interesting compare Proposition \ref{prop:1}
with the following result.
Let $\varphi \in {\mathcal S}$ be chosen
so that 
$\varphi(\xi)=1$ if $2 \le |\xi| \le 4$
and that
$\varphi(\xi)=0$ if $|\xi| \le 1$ or if $|\xi| \ge 8$.
Then in \cite{Sawano-Sugano-Tanaka-2009-Proc},
we have established
\begin{equation}\label{eq:100816-1}
\left\|\sum_{j=-\infty}^\infty
|2^{j\alpha}{\mathcal F}^{-1}[\varphi(2^{-j}\cdot){\mathcal F}f]|
\right\|_{{\mathcal M}^s_t}
\lesssim
\|f\|_{{\mathcal M}^p_q}.
\end{equation}
Estimate {\rm (\ref{eq:100816-1})} admits an extension
to Triebel-Lizorkin-Morrey spaces as we did
in \cite{Sawano-Sugano-Tanaka-2009-Proc}.
\end{remark}

\subsection{Proof of Theorem \ref{thm4}}
\label{subsection:thm4}

We freeze $\varepsilon \in E$ for a while.
By definition of $I_{\alpha,{\rm dyadic}}f$,
we obtain
\begin{align*}
\lefteqn{
\sum_{j=-\infty}^\infty \sum_{Q \in {\mathcal D}_j}
\langle a,h^{\varepsilon}_Q \rangle 
\{h^{\varepsilon}_Q(x) I_{\alpha,{\rm dyadic}}f(x)
-
I_{\alpha,{\rm dyadic}}[h^{\varepsilon}_Q f](x)\}
}\\
&=
\sum_{j=-\infty}^\infty \sum_{Q \in {\mathcal D}_j}
\sum_{\varepsilon' \in E}
\sum_{l=-\infty}^\infty\sum_{R \in {\mathcal D}_l}
\langle a,h^{\varepsilon}_Q \rangle
\langle f,h^{\varepsilon'}_R \rangle
\{h^{\varepsilon}_Q(x) I_{\alpha,{\rm dyadic}}h^{\varepsilon'}_R(x)
-
I_{\alpha,{\rm dyadic}}[h^{\varepsilon}_Q h^{\varepsilon'}_R](x)\}.
\end{align*}
Observe that, if $Q \supsetneq R$,
then
$h^{\varepsilon}_Q h^{\varepsilon'}_R=
m_R(h^{\varepsilon}_Q) h^{\varepsilon'}_R$
and hence
\begin{equation}
h^{\varepsilon}_Q(x) I_{\alpha,{\rm dyadic}}h^{\varepsilon'}_R(x)
=
I_{\alpha,{\rm dyadic}}[h^{\varepsilon}_Q h^{\varepsilon'}_R](x).
\end{equation}
Therefore,
we have
\begin{align*}
\lefteqn{
\sum_{j=-\infty}^\infty \sum_{Q \in {\mathcal D}_j}
\langle a,h^{\varepsilon}_Q \rangle 
\{
h^{\varepsilon}_Q(x) I_{\alpha,{\rm dyadic}}f(x)
-
I_{\alpha,{\rm dyadic}}[h^{\varepsilon}_Q f](x)
\}
}\\
&=
\sum_{j=-\infty}^\infty \sum_{Q \in {\mathcal D}_j}
\sum_{\varepsilon' \in E}
\sum_{l=-\infty}^j
\sum_{R \in {\mathcal D}_l}
\langle a,h^{\varepsilon}_Q \rangle
\langle f,h^{\varepsilon'}_R \rangle
\{
h^{\varepsilon}_Q(x) I_{\alpha,{\rm dyadic}}h^{\varepsilon'}_R(x)
-
I_{\alpha,{\rm dyadic}}[h^{\varepsilon}_Q h^{\varepsilon'}_R](x)
\}.
\end{align*}
If $Q=R$ and $\varepsilon \ne \varepsilon'$,
then we obtain
\begin{equation}
h^{\varepsilon}_Q(x) I_{\alpha,{\rm dyadic}}h^{\varepsilon'}_R(x)
=
I_{\alpha,{\rm dyadic}}[h^{\varepsilon}_Q h^{\varepsilon'}_R](x).
\end{equation}
Let us write
\begin{align*}
{\rm I}_1(x)
&:=
\sum_{j=-\infty}^\infty \sum_{Q \in {\mathcal D}_j}
\sum_{\varepsilon' \in E}
\sum_{l=-\infty}^{j-1}
\sum_{R \in {\mathcal D}_l}
|R|^{\frac{\alpha}{n}}
\langle a,h^{\varepsilon}_Q \rangle
\langle f,h^{\varepsilon'}_R \rangle
h^{\varepsilon}_Q(x)h^{\varepsilon'}_R(x),\\
{\rm I}_2(x)
&:=
\sum_{j=-\infty}^\infty \sum_{Q \in {\mathcal D}_j}
\sum_{\varepsilon' \in E}
\sum_{l=-\infty}^{j-1}
\sum_{R \in {\mathcal D}_l}
|Q|^{\frac{\alpha}{n}}
\langle a,h^{\varepsilon}_Q \rangle
\langle f,h^{\varepsilon'}_R \rangle
h^{\varepsilon}_Q(x)h^{\varepsilon'}_R(x),\\
{\rm II}(x)
&:=
\sum_{j=-\infty}^\infty \sum_{Q \in {\mathcal D}_j}
\langle a,h^{\varepsilon}_Q \rangle
\langle f,h^{\varepsilon}_Q \rangle
|Q|^{\frac{\alpha}{n}}|h^{\varepsilon}_Q(x)|^2,\\
{\rm III}(x)
&:=
\sum_{j=-\infty}^\infty \sum_{Q \in {\mathcal D}_j}
\langle a,h^{\varepsilon}_Q \rangle
\langle f,h^{\varepsilon}_Q \rangle
I_{\alpha,{\rm dyadic}}[|h^{\varepsilon}_Q|^2](x).
\end{align*}
Note that both ${\rm I}_1$ and ${\rm III}$ have
another expression:
\begin{align*}
{\rm I}_1(x)
&=
\sum_{j=-\infty}^\infty
\left(
\sum_{Q \in {\mathcal D}_j}
\langle I_{\alpha,{\rm dyadic}}f,\chi_Q \rangle 
\cdot
\langle a,h^{\varepsilon}_Q \rangle
\frac{h^{\varepsilon}_Q}{|Q|}
\right),\\
{\rm III}(x)
&=
I_{\alpha,{\rm dyadic}}
\left[
\sum_{j=-\infty}^\infty \sum_{Q \in {\mathcal D}_j}
\langle a,h^{\varepsilon}_Q \rangle
\langle f,h^{\varepsilon}_Q \rangle|h^{\varepsilon}_Q|^2
\right](x).
\end{align*}
Hence,
we have
\begin{align*}
\lefteqn{
\sum_{j=-\infty}^\infty \sum_{Q \in {\mathcal D}_j}
\langle a,h^{\varepsilon}_Q \rangle 
(h^{\varepsilon}_Q(x) I_{\alpha,{\rm dyadic}}f(x)
-
I_{\alpha,{\rm dyadic}}[h^{\varepsilon}_Q f](x))
}\\
&=
\sum_{j=-\infty}^\infty \sum_{Q \in {\mathcal D}_j}
\sum_{\varepsilon' \in E}
\sum_{l=-\infty}^{j-1}
\sum_{R \in {\mathcal D}_l}
\langle a,h^{\varepsilon}_Q \rangle
\langle f,h^{\varepsilon'}_R \rangle
\{
h^{\varepsilon}_Q(x) I_{\alpha,{\rm dyadic}}h^{\varepsilon'}_R(x)
-
I_{\alpha,{\rm dyadic}}[h^{\varepsilon}_Q h^{\varepsilon'}_R](x)
\}
\\
&\quad+
\sum_{j=-\infty}^\infty \sum_{Q \in {\mathcal D}_j}
\langle a,h^{\varepsilon}_Q \rangle
\langle f,h^{\varepsilon}_Q \rangle
\{
h^{\varepsilon}_Q(x) I_{\alpha,{\rm dyadic}}h^{\varepsilon}_Q(x)
-
I_{\alpha,{\rm dyadic}}[h^{\varepsilon}_Q h^{\varepsilon}_Q](x)
\}\\
&=
{\rm I}_1(x)-{\rm I}_2(x)+{\rm II}(x)-{\rm III}(x).
\end{align*}

Let us start with dealing with ${\rm I}_1$.
If we invoke again Theorem \ref{thm2},
then we have
\begin{align*}
\|{\rm I}_1\|_{{\mathcal M}^s_t}
\lesssim
\|a\|_{{\rm BMO}_{{\rm dyadic}}}
\|I_{\alpha,{\rm dyadic}}f\|_{{\mathcal M}^s_t}
\lesssim
\|a\|_{{\rm BMO}_{{\rm dyadic}}}
\|f\|_{{\mathcal M}^p_q}.
\end{align*}
Since
\begin{align*}
{\rm I}_2(x)
&=
\sum_{\varepsilon' \in E}
\sum_{j=-\infty}^\infty 
\sum_{l=-\infty}^{j-1}
\sum_{R \in {\mathcal D}_l}
\sum_{Q \in {\mathcal D}_j, \, Q \subsetneq R}
|R|^{\frac{\alpha}{n}}
\langle a,h^{\varepsilon}_Q \rangle
\langle f,h^{\varepsilon'}_R \rangle
h^{\varepsilon}_Q(x)h^{\varepsilon'}_R(x)\\
&=
\sum_{j=-\infty}^\infty 
\sum_{Q \in {\mathcal D}_j}
\langle a,h^{\varepsilon}_Q \rangle
\langle I_{\alpha,{\rm dyadic}}f,\chi_Q \rangle
h^{\varepsilon}_Q(x),
\end{align*}
we have by Theorem \ref{thm2}
\begin{align*}
\|{\rm I}_2\|_{{\mathcal M}^s_t}
&=
\left\|
\sum_{j=-\infty}^\infty 
\sum_{Q \in {\mathcal D}_j}
\langle a,h^{\varepsilon}_Q \rangle
\langle I_{\alpha,{\rm dyadic}}f,\chi_Q \rangle
h^{\varepsilon}_Q
\right\|_{{\mathcal M}^s_t}
\lesssim
\|a\|_{{\rm BMO}_{{\rm dyadic}}}
\|I_{\alpha,{\rm dyadic}}f\|_{{\mathcal M}^s_t}
\end{align*}
If we invoke Theorem \ref{thm3},
then we have
\begin{align*}
\|{\rm I}_2\|_{{\mathcal M}^s_t}
\lesssim
\|a\|_{{\rm BMO}_{{\rm dyadic}}}\|f\|_{{\mathcal M}^p_q}.
\end{align*}

By Proposition \ref{prop:1}
we obtain
\begin{align*}
\|{\rm II}\|_{{\mathcal M}^s_t}
&=
\left\|\sum_{j=-\infty}^\infty 
\left|\sum_{Q \in {\mathcal D}_j}|Q|^{\frac{\alpha}{n}}
\langle a,h^{\varepsilon}_Q \rangle
\langle f,h^{\varepsilon}_Q \rangle
h^{\varepsilon}_Q \cdot h^{\varepsilon}_Q
\right|\,\right\|_{{\mathcal M}^s_t}\\
&\le
\|a\|_{{\rm BMO}_{{\rm dyadic}}}
\left\|\sum_{j=-\infty}^\infty 
\left|\sum_{Q \in {\mathcal D}_j}
|Q|^{\frac{\alpha}{n}}\langle f,h^{\varepsilon}_Q \rangle
h^{\varepsilon}_Q
\right|\,\right\|_{{\mathcal M}^s_t}\\
&\lesssim
\|a\|_{{\rm BMO}_{{\rm dyadic}}}\|f\|_{{\mathcal M}^p_q}.
\end{align*}
Next,
by Proposition \ref{prop:1} and equality
$\displaystyle
I_{\alpha,{\rm dyadic}}[h^{\varepsilon}_Q h^{\varepsilon}_Q](x)
=
c_{\alpha}|Q|^{\frac{\alpha}{n}-1}\chi_Q(x),
$
we have
\begin{align*}
\|{\rm III}\|_{{\mathcal M}^s_t}
\lesssim
\left\|
\sum_{j=-\infty}^\infty \sum_{Q \in {\mathcal D}_j}
|\langle a,h^{\varepsilon}_Q \rangle
\langle f,h^{\varepsilon}_Q \rangle
|Q|^{\frac{\alpha}{n}-1}|\chi_Q
\right\|_{{\mathcal M}^s_t}
\lesssim
\|a\|_{{\rm BMO}_{{\rm dyadic}}}\|f\|_{{\mathcal M}^p_q}.
\end{align*}
Thus, the proof of Theorem \ref{thm4} is complete.

\subsection{Proof of Theorem \ref{thm5}}
\label{subsection:thm5}

Let $U \in {\mathcal D}$ be fixed.
Then we have
\begin{align*}
[a,I_{\alpha,{\rm dyadic}}]h^{\varepsilon''}_U(x)
&=
\sum_{\varepsilon \in E}
\sum_{j=-\infty}^\infty 
\sum_{Q \in {\mathcal D}_j, \, Q \subsetneq U}
(|Q|^{\frac{\alpha}{n}}-|U|^{\frac{\alpha}{n}})
\langle a,h^{\varepsilon}_Q \rangle
h^{\varepsilon}_Q(x)h^{\varepsilon''}_U(x)\\
&\quad+
\langle a,h^{\varepsilon''}_U \rangle
\langle f,h^{\varepsilon''}_U \rangle
(|U|^{\frac{\alpha}{n}}h^{\varepsilon''}_U(x)h^{\varepsilon''}_U(x)
-
I_{\alpha,{\rm dyadic}}[h^{\varepsilon''}_Q h^{\varepsilon''}_U](x)).
\end{align*}
By virtue of the non-homogeneous wavelet expansion
(see (\ref{eq:100104-2})),
we obtain
\begin{align*}
\|\,[a,I_{\alpha,{\rm dyadic}}]\,\|_{B({\mathcal M}^s_t,{\mathcal M}^p_q)}
&\ge
\|\,[a,I_{\alpha,{\rm dyadic}}]h^{\varepsilon''}_U\,\|_{{\mathcal M}^s_t}
|U|^{-\frac1p+\frac{1}{2}}\\
&\gtrsim
\left\|
\sum_{j=-\infty}^\infty 
\sum_{Q \in {\mathcal D}_j, \, Q \subsetneq U}
(|Q|^{\frac{\alpha}{n}}-|U|^{\frac{\alpha}{n}})
\langle a,h^{\varepsilon}_Q \rangle
h^{\varepsilon}_Q 
\right\|_{{\mathcal M}^s_t}|U|^{-\frac1p}.
\end{align*}
By Theorem \ref{thm1}
we have
\begin{align*}
\|\,[a,I_{\alpha,{\rm dyadic}}]\,\|_{B({\mathcal M}^s_t,{\mathcal M}^p_q)}
&\gtrsim
|U|^{\frac{\alpha}{n}-\frac1p}
\left\|
\sum_{j=-\infty}^\infty 
\sum_{Q \in {\mathcal D}_j, \, Q \subsetneq U}
\langle a,h^{\varepsilon}_Q \rangle
h^{\varepsilon}_Q h^{\varepsilon''}_U
\right\|_{{\mathcal M}^s_t}\\
&=
\frac{1}{|U|^{\frac{1}{s}}}
\left\|
\sum_{j=-\infty}^\infty 
\sum_{Q \in {\mathcal D}_j, \, Q \subsetneq U}
\langle a,h^{\varepsilon}_Q \rangle
h^{\varepsilon}_Q 
\right\|_{{\mathcal M}^s_t}\\
&\ge
\frac{1}{|U|^{\frac1t}}
\left(\int_U
\left|
\sum_{j=-\infty}^\infty 
\sum_{Q \in {\mathcal D}_j, \, Q \subsetneq U}
\langle a,h^{\varepsilon}_Q \rangle
h^{\varepsilon}_Q(x) \right|^t\,dx\right)^{\frac1t}.
\end{align*}
By the H\"{o}lder inequalty
we have 
\begin{align*}
\frac{1}{|U|}
\int_U\left|
\sum_{j=-\infty}^\infty 
\sum_{Q \in {\mathcal D}_j, \, Q \subsetneq U}
\langle a,h^{\varepsilon}_Q \rangle
h^{\varepsilon}_Q(x) \right|\,dx
\lesssim
\|\,[a,I_{\alpha,{\rm dyadic}}]\,\|_{B({\mathcal M}^s_t,{\mathcal M}^p_q)}.
\end{align*}
Therefore,
we obtain
\begin{equation}
\label{eq:5.11}
m_U\left(
\left|
\sum_{j=-\infty}^\infty 
\sum_{Q \in {\mathcal D}_j, \, Q \subsetneq U}
\langle a,h^{\varepsilon}_Q \rangle
h^{\varepsilon}_Q\right|\right)
\lesssim
\frac{
\|\,[a,I_{\alpha,{\rm dyadic}}]h^{\varepsilon''}_U\,\|_{{\mathcal M}^s_t}
}{\|h^{\varepsilon''}_U\|_{{\mathcal M}^p_q}}
\end{equation}
for all $U \in {\mathcal D}$.
Denote by $U^*$ the dyadic parent of $U$,
that is, the smallest dyadic cube engulfing $U$.
With $U$ replaced by $U^*$ above,
we obtain 
\begin{equation}
\label{eq:5.12}
m_U\left(
\left|
\sum_{j=-\infty}^\infty 
\sum_{Q \in {\mathcal D}_j, \, Q \subset U}
\langle a,h^{\varepsilon}_Q \rangle
h^{\varepsilon}_Q\right|\right)
\lesssim
\frac{
\|\,[a,I_{\alpha,{\rm dyadic}}]h^{\varepsilon''}_U\,\|_{{\mathcal M}^s_t}
}{\|h^{\varepsilon''}_U\|_{{\mathcal M}^p_q}}.
\end{equation}
It follows from the definition of the operator norm
$\|\,[a,I_{\alpha,{\rm dyadic}}]\,\|_{{\mathcal M}^p_q \to {\mathcal M}^s_t}$
that
\begin{equation}
\label{eq:5.13}
m_U\left(
\left|
\sum_{j=-\infty}^\infty 
\sum_{Q \in {\mathcal D}_j, \, Q \subset U}
\langle a,h^{\varepsilon}_Q \rangle
h^{\varepsilon}_Q\right|\right)
\lesssim
\|\,[a,I_{\alpha,{\rm dyadic}}]\,\|_{{\mathcal M}^p_q \to {\mathcal M}^s_t}.
\end{equation}
If we take the supremum over $U \in {\mathcal D}$
in (\ref{eq:5.13}),
then we obtain
\begin{align*}
\|a\|_{{\rm BMO}_{{\rm dyadic}}}
&\lesssim
\|\,[a,I_{\alpha,{\rm dyadic}}]\,\|_{B({\mathcal M}^s_t,{\mathcal M}^p_q)}.
\end{align*}
Thus, the proof is complete.

\subsection{Proof of Theorem \ref{thm6}}
\label{subsection:thm6}

By Theorem \ref{thm1} $(iii)$,
we see that $\{h^\varepsilon_Q\}_{Q \in {\mathcal D},\varepsilon \in E}$
is dense in ${\mathcal H}^{r_0'}_{r'}$.
Thus, to check $(i)$ and $(ii)$,
we need only to prove
\[
\left\|I_{\alpha,{\rm dyadic}}f\right\|_{{\mathcal H}^{p_0'}_{p'}}
+
\left\|[a,I_{\alpha,{\rm dyadic}}]f\right\|_{{\mathcal H}^{p_0'}_{p'}}
\lesssim
\|f\|_{{\mathcal H}^{r_0'}_{r'}}
\]
for all 
$f \in {\rm Span}(\{h^\varepsilon_Q\}_{Q \in {\mathcal D},\varepsilon \in E})$.
If we assume 
$f \in {\rm Span}(\{h^\varepsilon_Q\}_{Q \in {\mathcal D},\varepsilon \in E})$,
then from the definition we have
\begin{equation}\label{eq:100818-1}
I_{\alpha,{\rm dyadic}}f,[a,I_{\alpha,{\rm dyadic}}]f
\in{\mathcal H}^{p_0'}_{p'}.
\end{equation}
Observe that (\ref{eq:100818-1}) counts
in that we can obtain (\ref{eq:100818-1})
only by using the discrete fractional integral operators.
Consequently, if we invoke Theorem \ref{thm1},
we obtain
\begin{align*}
\left\|I_{\alpha,{\rm dyadic}}f\right\|_{{\mathcal H}^{p_0'}_{p'}}
&=
\sup_{g \in {\mathcal M}^{p_0}_{p} \setminus \{0\}}
\frac{1}{\|g\|_{{\mathcal M}^{p_0}_{p}}}
\left|\int_{\R^n}g(x)I_{\alpha,{\rm dyadic}}f(x)\,dx\right|\\
&=
\sup_{g \in {\mathcal M}^{p_0}_{p} \setminus \{0\}}
\frac{1}{\|g\|_{{\mathcal M}^{p_0}_{p}}}
\left|\int_{\R^n}I_{\alpha,{\rm dyadic}}g(x)f(x)\,dx\right|\\
&\le
\sup_{g \in {\mathcal M}^{p_0}_{p} \setminus \{0\}}
\frac{1}{\|g\|_{{\mathcal M}^{p_0}_{p}}}
\|I_{\alpha,{\rm dyadic}}g\|_{{\mathcal M}^{r_0}_{r}}
\|f\|_{{\mathcal H}^{r_0'}_{r'}}\\
&\le
\|I_{\alpha,{\rm dyadic}}\|_{{\mathcal M}^{p_0}_{p} \to {\mathcal M}^{r_0}_r}
\|f\|_{{\mathcal H}^{r_0'}_{r'}}.
\end{align*}
Thus, the proof is now complete.

\begin{remark}\label{rem:1}
A usual averaging procedure yields the following corollaries.
For this technique we refer to \cite{Lacey-2005-Hokkaido}.
\begin{corollary}\label{cor:3.6}
Maintain the same conditions
on the parameters $p,p_0,r,r_0,\alpha$.
If a function $a$ belongs to ${\rm BMO}$,
then the following boundedness is true{\rm:}
\[
\left\|I_{\alpha}f\right\|_{{\mathcal H}^{p_0'}_{p'}}
+
\left\|[a,I_{\alpha}]f\right\|_{{\mathcal H}^{p_0'}_{p'}}
\lesssim
\|f\|_{{\mathcal H}^{r_0'}_{r'}},
\]
where $I_{\alpha}$ is given by
\[
I_{\alpha}f(x)=\int_{{\mathbb R}^n}
\frac{f(y)}{|x-y|^{n-\alpha}}\,dy.
\]
\end{corollary}
As for $I_\alpha$ we made an alternative approach
in \cite[Theorem 3.1]{Sawano-Sugano-Tanaka-2009-BVP}
and \cite[Theorem 3.1]{Sawano-Sugano-Tanaka-2010-TAMS}.
\end{remark}

\subsection{Proof of Theorem \ref{thm7}}
\label{subsection:thm7}

\lq \lq If part" is a direct consequence of Theorem \ref{thm5}.
Let us prove the converse.
To this end,
we need the following fundamental lemma.
\begin{lemma}\label{lem3.3}
Let $X$ and $Y$ be Banach spaces.
Suppose that we are given a compact linear operator $T:X \to Y$.
If $\{f_j\}_{j \in \N}$ is a sequence in $X^*$
that is weak-* convergent to $0$.
Then $\{T^*f_j\}_{j \in \N}$ is norm-convergent to $0$.
\end{lemma}

Now let us prove $a \in {\rm VMO}_{{\rm dyadic}}$ 
assuming that $[a,I_{\alpha,{\rm dyadic}}]$ is compact.

Let us set
\begin{equation}
[a,I_{\alpha,{\rm dyadic}}]_{\ge L}
:=
\lim_{M \to \infty}
\sum_{\varepsilon \in E}
\sum_{j=L}^M
\sum_{Q \in {\mathcal D}_j}
\langle a \cdot I_{\alpha,{\rm dyadic}}f-I_{\alpha,{\rm dyadic}}[a \cdot f],
h^{\varepsilon}_Q \rangle h^{\varepsilon}_Q.
\end{equation}
Then we have
\begin{equation}
\|\,[a,I_{\alpha,{\rm dyadic}}]_{\ge L}\,
\|_{{\mathcal M}^p_q \to {\mathcal M}^s_t}
\lesssim
\sum_{\varepsilon \in E}
\sup_{U \in {\mathcal D}}
m_U\left(
\left|
\sum_{j=L}^{\infty}
\sum_{Q \in {\mathcal D}_j, \, Q \subset U}
\langle a,h^{\varepsilon}_Q \rangle
h^{\varepsilon}_Q\right|\right)
\end{equation}
by virtue of Theorem \ref{thm4}.
The triangle inequality yields
\begin{align*}
\sup_{U \in {\mathcal D}}
m_U\left(
\left|
\sum_{j=L}^{\infty}
\sum_{Q \in {\mathcal D}_j, \, Q \subset U}
\langle a,h^{\varepsilon}_Q \rangle
h^{\varepsilon}_Q\right|\right)
&=
\sup_{U \in \bigcup_{\nu=L}^\infty {\mathcal D}_\nu}
m_U\left(
\left|
\sum_{j=L}^{\infty}
\sum_{Q \in {\mathcal D}_j, \, Q \subset U}
\langle a,h^{\varepsilon}_Q \rangle
h^{\varepsilon}_Q\right|\right).
\end{align*}
According to (\ref{eq:5.12}) and Lemma \ref{lem3.3},
we have
\begin{align*}
\lim_{L \to \infty}
\sup_{U \in {\mathcal D}}
m_U\left(
\left|
\sum_{j=L}^{\infty}
\sum_{Q \in {\mathcal D}_j, \, Q \subset U}
\langle a,h^{\varepsilon}_Q \rangle
h^{\varepsilon}_Q\right|\right)
&=0.
\end{align*}
Thus, assuming that $[a,I_{\alpha,{\rm dyadic}}]$ is compact,
we have
\begin{equation}
\label{eq:39}
\lim_{L \to \infty}
[a,I_{\alpha,{\rm dyadic}}]_{\ge L}=0
\end{equation}
in the operator topology.
Also, we set
\begin{equation}
[a,I_{\alpha,{\rm dyadic}}]_{\le -L}
:=
\lim_{M \to \infty}
\sum_{\varepsilon \in E}
\sum_{j=-M}^{-L}
\sum_{Q \in {\mathcal D}_j}
\langle a \cdot I_{\alpha,{\rm dyadic}}f-I_{\alpha,{\rm dyadic}}[a \cdot f],
h^{\varepsilon}_Q \rangle h^{\varepsilon}_Q.
\end{equation}
Then we have
\begin{equation}
\|\,[a,I_{\alpha,{\rm dyadic}}]_{\le -L}\,
\|_{{\mathcal M}^p_q \to {\mathcal M}^s_t}
\lesssim
\sum_{\varepsilon \in E}
\sup_{U \in {\mathcal D}}
m_U\left(
\left|
\sum_{j=-\infty}^{-L}
\sum_{Q \in {\mathcal D}_j, \, Q \subset U}
\langle a,h^{\varepsilon}_Q \rangle
h^{\varepsilon}_Q\right|\right)
\end{equation}
by virtue of Theorem \ref{thm4}.
Note that
\begin{align*}
\lefteqn{
\sup_{U \in {\mathcal D}}
m_U\left(
\left|
\sum_{j=-\infty}^{-L}
\sum_{Q \in {\mathcal D}_j, \, Q \subset U}
\langle a,h^{\varepsilon}_Q \rangle
h^{\varepsilon}_Q\right|\right)
}\\
&=
\sup\left\{
m_U\left(
\left|
\sum_{j=-\infty}^{-L}
\sum_{Q \in {\mathcal D}_j, \, Q \subset U}
\langle a,h^{\varepsilon}_Q \rangle
h^{\varepsilon}_Q\right|\right)
\,:\,U \in \bigcup_{\nu=-\infty}^{-L}{\mathcal D}_\nu\right\}
\end{align*}
in order that $Q \subset U$ actually happens.
Thus, assuming that $[a,I_{\alpha,{\rm dyadic}}]$ is compact,
we have
\begin{equation}
\label{eq:42}
\lim_{L \to \infty}
[a,I_{\alpha,{\rm dyadic}}]_{\le -L}=0
\end{equation}
again in the opertor topology.
{}From (\ref{eq:39}) and (\ref{eq:42})
we have
\begin{equation}
\label{eq:43}
\lim_{L \to \infty}
\|a-a_{(L)}\|_{{\rm BMO}_{{\rm dyadic}}}
=0,
\end{equation}
if we write
\[
a_{(L)}:=\sum_{\varepsilon \in E}
\sum_{j=-L}^L
\sum_{Q \in {\mathcal D}_j}
\langle a,h^\varepsilon_Q \rangle h^\varepsilon_Q.
\]
It is not so hard to see that
\begin{equation}
\lim_{R \to \infty}\sup
\{
|\langle a,h^\varepsilon_{Q_{\nu m}} \rangle|
\,:\,m \in \Z^n, \, |m| \ge R\}=0
\end{equation}
for all $\nu \in \Z$ if $[a,I_{\alpha,{\rm dyadic}}]$ is compact.
Thus, it follows that
\begin{equation}
\label{eq:45}
\lim_{R \to \infty}
\left\|
\sum_{m \in {\mathbb Z}^n, \, |m|>R}
\langle a,h^\varepsilon_{Q_{jm}} \rangle h^\varepsilon_{Q_{jm}}
\right\|_{{\rm BMO}_{{\rm dyadic}}}=0
\end{equation}
for all $j \in {\mathbb Z}$.

{}From (\ref{eq:45})
we learn that $a_{(L)} \in {\rm VMO}_{{\rm dyadic}}$,
which in turn yields $a \in {\rm VMO}_{{\rm dyadic}}$ 
by virtue of (\ref{eq:43}).

\section{Acknowledgement}

The author is supported 
by Grant-in-Aid for Young Scientists (B) No. 24740085
from the 
Japan Society for the Promotion of Science. 
The author is also indebted to Zenis Co. Ltd
for the check of the presentation in English
in Section 1.

\begin{flushleft}

Yoshihiro Sawano\\
Department of Mathematics and Information Sciences, \\
Tokyo Metropolitan University, \\
1-1 Minami-Ohsawa, Hachioji, Tokyo, 192-0397, Japan\\
E-mail : ysawano@tmu.ac.jp \\

\end{flushleft}
\end{document}